\documentclass[reqno]{amsart}
\usepackage{amsmath}

\usepackage{graphicx}
\usepackage{amscd}
\usepackage{amsmath, amssymb}
\usepackage{graphics}
\usepackage{graphicx}
\usepackage{epsfig}
\usepackage{xcolor}
\usepackage{soul}
\usepackage{tikz}

\newtheorem{theorem}{Theorem}

\newtheorem{corollary}[theorem]{Corollary}

\newtheorem{definition}[theorem]{Definition}

\newtheorem{lemma}[theorem]{Lemma}

\newtheorem{proposition}[theorem]{Proposition}

\newtheorem{question}[theorem]{Question}

\begin{document}

\title[On the irreducibility of $SB_n$ representations]{On the irreducibility of singular braid group representations}

\author{Mohamad N. Nasser}

\address{Mohamad N. Nasser\\
         Department of Mathematics and Computer Science\\
         Beirut Arab University\\
         P.O. Box 11-5020, Beirut, Lebanon}
         
\email{m.nasser@bau.edu.lb}

\begin{abstract}
In this article, we investigate the irreducibility of representations of the singular braid group on $n$ strands, denoted by $SB_n$. We first show that all representations $\rho: SB_2 \longrightarrow \mathrm{GL}_2(\mathbb{C})$ are reducible. We then determine all irreducible representations $\rho: SB_3 \longrightarrow \mathrm{GL}_2(\mathbb{C})$ and construct several families of irreducible representations $\rho: SB_3 \longrightarrow \mathrm{GL}_3(\mathbb{C})$. Finally, we study the irreducibility of homogeneous local representations of $SB_n$ for all $n \geq 3$.
\end{abstract}

\maketitle

\renewcommand{\thefootnote}{}
\footnote{\textit{Key words and phrases.} Braid Group; Singular Braid Group; Group Representations; Irreducible Representations.}
\footnote{\textit{Mathematics Subject Classification.} Primary: 20F36.}

\vspace*{-0.4cm}

\section{Introduction} 

The braid group on $n$ strands, denoted by $B_n$, is one of the most fundamental algebraic structures in mathematics, with deep connections to algebra, topology, and geometry. It was introduced by E. Artin in \cite{E.A} as the group of equivalence classes of braids, where a braid consists of $n$ vertical strands that may cross over or under one another, and two braids are considered equivalent if they are related by a continuous deformation. Since its introduction, the braid group has found numerous applications across mathematics and physics. In knot theory, it provides a framework for studying links through braid closures \cite{K.Mu}; in algebraic geometry, it plays a central role in the study of monodromy \cite{D.Co}; in mathematical physics, it arises in the description of anyons \cite{Y.Ch}; and in quantum computing, it serves as a foundation for topological quantum computation \cite{C.DE}.

\vspace*{0.1cm}

The singular braid monoid on $n$ strands, denoted by $SM_n$, is an extension of the braid group $B_n$ obtained by adjoining a family of singular generators $\tau_1,\tau_2,\ldots,\tau_{n-1}$ to the classical Artin generators $\sigma_1,\sigma_2,\ldots,\sigma_{n-1}$. It was introduced by J. Birman in \cite{J.Bir}. Later, R. Fenn et al. proved in \cite{R.Fe} that $SM_n$ embeds into a group, called the singular braid group and denoted by $SB_n$, which is defined by the same generators and relations as $SM_n$. By incorporating singular crossings, $SB_n$ extends the classical braid group framework and has become an important object of study in several areas of mathematics. For further details on the singular braid monoid and singular braid group, we refer the reader to \cite{O.Da,O.Das,B.Gem}.

\vspace*{0.1cm}

The classification of irreducible representations of a group is a fundamental problem in representation theory, particularly for the braid group $B_n$ and its extensions. Irreducible representations reveal important structural properties of groups and play a significant role in several areas of mathematics and physics, including quantum mechanics and particle physics (see, for example, \cite{J.Yang}). In \cite{E.F}, E. Formanek classified all irreducible complex representations $\rho: B_n \longrightarrow \mathrm{GL}_r(\mathbb{C})$ for $n \geq 7$ and $2 \leq r \leq n-1$. This result was later generalized by M. Nasser et al. in \cite{M.R.A} to the group of conjugating automorphisms of a free group, $C_n$, which is an extension of $B_n$. More precisely, they classified all irreducible complex representations $\rho: C_n \longrightarrow \mathrm{GL}_r(\mathbb{C})$ for $n \geq 7$ and $2 \leq r \leq n-3$. In a different direction, I. Sysoeva classified several families of irreducible complex representations of the welded braid group $WB_n$, which is isomorphic to $C_n$ \cite{I.Syso}.

\vspace*{0.1cm}

Regarding the irreducibility of representations of $SB_n$, M. Nasser investigated in \cite{MMM1} the irreducibility of extensions to $SB_n$ of the Wada representation of $B_n$ introduced in \cite{wada}. He also studied in \cite{MMM2} the irreducibility of representations of $SB_n$ extending the standard representation of $B_n$ defined in \cite{D.T}. Motivated by these works, the present paper investigates the irreducibility of representations of $SB_n$ from a broader perspective.

\vspace*{0.1cm}

In Section 2, we give an overview for the targeted groups in this paper, $B_n$ and $SB_n$, and we present preliminaries regarding our results. In section 3, we introduce our first result, where we show that all representations $\rho: SB_2 \longrightarrow \mathrm{GL}_2(\mathbb{C})$ are reducible. Our second result is in Section 4, where we classify the forms of all irreducible representations $\rho: SB_3 \longrightarrow \mathrm{GL}_2(\mathbb{C})$. Also, we present the forms of some irreducible representations $\rho: SB_3 \longrightarrow \mathrm{GL}_3(\mathbb{C})$. In Section 5, we present our third result, where we study the irreducibility of the homogeneous local representations of $SB_n$ for all $n\geq 3$. More deeply, we prove that every homogeneous local representation of $SB_n$ having one of the forms $\rho_1$ and $\rho_2$ is reducible. Also, we find necessary and sufficient conditions for the homogeneous local representations of $SB_n$ having the form $\rho_3$ to be irreducible.

\vspace*{0.15cm}

\section{Overview} 

First, we give the presentation of the braid group $B_n$ introduced by E. Artin.

\begin{definition} \cite{E.A}
The braid group on $n$ strands, $B_n$, is the group defined by the Artin generators $\sigma_1,\sigma_2,\ldots,\sigma_{n-1}$ with the following defining relations.
\begin{align}
\sigma_i\sigma_{i+1}\sigma_i &= \sigma_{i+1}\sigma_i\sigma_{i+1},
& 1\leq i\leq n-2,\label{eqs1}\\
\sigma_i\sigma_j &= \sigma_j\sigma_i,
& |i-j|\geq 2.\label{eqs2}
\end{align}
\end{definition}

Now, we give the presentation of the singular braid monoid $SM_n$ introduced by J. Birman.

\begin{definition} \cite{J.Bir}
The singular braid monoid on $n$ strands, $SM_n$, is the monoid defined by the non-singular generators $\sigma_1,\sigma_2, \ldots,\sigma_{n-1}$ of $B_n$ in addition to the singular generators $\tau_1,\tau_2, \ldots, \tau_{n-1}$. In addition to the relations \eqref{eqs1} and \eqref{eqs2} above, the generators $\sigma_i$ and $\tau_i$ of $SM_n$ satisfy the following relations.
\begin{align}
\tau_i\tau_j &= \tau_j\tau_i,
& |i-j|\geq 2,\label{eqs3}\\
\tau_i\sigma_j &= \sigma_j\tau_i,
& |i-j|\geq 2,\label{eqs4}\\
\tau_i\sigma_i &= \sigma_i\tau_i,
& 1\leq i\leq n-1,\label{eqs5}\\
\sigma_i\sigma_{i+1}\tau_i &= \tau_{i+1}\sigma_i\sigma_{i+1},
& 1\leq i\leq n-2,\label{eqs6}\\
\sigma_{i+1}\sigma_i\tau_{i+1} &= \tau_i\sigma_{i+1}\sigma_i,
& 1\leq i\leq n-2.\label{eqs7}
\end{align}
\end{definition}

In \cite{R.Fe}, R. Fenn, E. Keyman, and C. Rourke proved that $SM_n$ embeds into the group $SB_n$, which is called the singular braid group and has the same generators and defining relations as $SM_n$. In the following figures, we show the geometrical shapes of the generators of $SB_n$.

\begin{center}
\begin{tikzpicture}
 	\draw[thick] (-1.5,0)--(-1.5,2); 
    \fill (-1,1) circle(1.5pt) (-1.2,1) circle(1.5pt)(-0.8,1)circle(1.5pt);       
    \draw[thick] (-.5,0)--(-.5,2);
    \draw[thick] (2.5,0)--(2.5,2); 
    \fill (2,1) circle(1.5pt) (2.2,1) circle(1.5pt)
    (1.8,1)circle(1.5pt);       
    \draw[thick] (1.5,0)--(1.5,2);
    \draw[thick] (0,0) to[out=90, in=-90] (1,2);
    \draw[thick, white, line width=4pt] (1,0) to[out=90, in=-90] (0,2); 
    \draw[thick] (1,0) to[out=90, in=-90] (0,2);
    \node[above] at(0,2){$i$};
    \node[above] at(1,2){$i+1$};
	\node at (3.5,1){,};  
 	\draw[thick] (5.5,0)--(5.5,2); 
    \fill (5,1) circle(1.5pt) (4.8,1) circle(1.5pt)(5.2,1)circle(1.5pt);       
    \draw[thick] (4.5,0)--(4.5,2);
    \draw[thick] (8.5,0)--(8.5,2); 
    \fill (8,1) circle(1.5pt) (8.2,1) circle(1.5pt)(7.8,1)circle(1.5pt);       
    \draw[thick] (7.5,0)--(7.5,2);    
    \draw[thick] (7,0) to[out=90, in=-90] (6,2);
    \draw[thick, white, line width=4pt] (6,0) to[out=90, in=-90] (7,2); 
    \draw[thick] (6,0) to[out=90, in=-90] (7,2);
    \node[above] at(6,2){$i$};
    \node[above] at (7,2) {$i+1$};
    \node at (3.5, -0.5) {The generators \(\sigma_i\) and \(\sigma_i^{-1}\)};
\end{tikzpicture}
\end{center}

\begin{center}
\begin{tikzpicture}
	\draw[thick] (5.5,0)--(5.5,2); 
    \fill (5,1) circle(1.5pt) (4.8,1) circle(1.5pt)(5.2,1)circle(1.5pt);       
    \draw[thick] (4.5,0)--(4.5,2);
    \draw[thick] (8.5,0)--(8.5,2); 
    \fill (8,1) circle(1.5pt) (8.2,1) circle(1.5pt)(7.8,1)circle(1.5pt);       
    \draw[thick] (7.5,0)--(7.5,2);    
	\draw[thick] (7,0) to[out=90, in=-90] (6,2);
    \draw[thick, white, line width=4pt] (6,0) to[out=90, in=-90] (7,2); 
    \draw[thick] (6,0) to[out=90, in=-90] (7,2);
    \draw[draw=black,fill=white] (6.5,1) circle (2pt); 
    \node[above] at(6,2){$i$};
    \node[above] at (7,2) {$i+1$};
	\node at (3.5,1){,};    
 	\draw[thick] (-1.5,0)--(-1.5,2); 
    \fill (-1,1) circle(1.5pt) (-1.2,1) circle(1.5pt)(-0.8,1)circle(1.5pt);       
    \draw[thick] (-.5,0)--(-.5,2);
    \draw[thick] (2.5,0)--(2.5,2); 
    \fill (2,1) circle(1.5pt) (2.2,1) circle(1.5pt)(1.8,1)circle(1.5pt);       
    \draw[thick] (1.5,0)--(1.5,2);
    \draw[thick] (0,0) to[out=90, in=-90] (1,2);
    \draw[thick, white, line width=4pt] (1,0) to[out=90, in=-90] (0,2); 
    \draw[thick] (1,0) to[out=90, in=-90] (0,2);
    \fill[black] (0.5,1) circle (2pt); 
    \node[above] at(0,2){$i$};
    \node[above] at(1,2){$i+1$};
    \node at (3.5, -0.5) {The generators \(\tau_i\) and \(\tau_i^{-1}\)};
\end{tikzpicture}
\end{center}
 
Now, we give, in the next two definitions, the concept of an important type of representation of $B_n$ and $SB_n$, namely local representations.

\begin{definition}
A representation $\mu: B_n \longrightarrow \mathrm{GL}_{m}(\mathbb{C})$ is called local if it has the form
$$\mu(\sigma_i) =\left( \begin{array}{c|@{}c|c@{}}
   \begin{matrix}
     I_{i-1} 
   \end{matrix} 
      & 0 & 0 \\
      \hline
    0 &\hspace{0.2cm} \begin{matrix}
   		M_i
   		\end{matrix}  & 0  \\
\hline
0 & 0 & I_{n-i-1}
\end{array} \right) \hspace*{0.2cm} \text{for} \hspace*{0.2cm} 1\leq i\leq n-1,$$ 
where $M_i \in \mathrm{GL}_k(\mathbb{C})$ with $k=m-n+2$ and $I_r$ is the $r\times r$ identity matrix. We say that the representation $\mu$ is homogeneous if all the matrices $M_i$ are equal.
\end{definition}

\begin{definition} \label{LocalSBn}
A representation $\rho: SB_n \longrightarrow \mathrm{GL}_{m}(\mathbb{C})$ is called local if it has the form
$$\rho(\sigma_i)=\left( \begin{array}{c|@{}c|c@{}}
   \begin{matrix}
     I_{i-1} 
   \end{matrix} 
      & 0 & 0 \\
      \hline
    0 &\hspace{0.2cm} \begin{matrix}
   		M_i
   		\end{matrix}  & 0  \\
\hline
0 & 0 & I_{n-i-1}
\end{array} \right) \hspace*{0.2cm} \text{for} \hspace*{0.2cm} 1\leq i\leq n-1$$
and
$$\rho(\tau_i) =\left( \begin{array}{c|@{}c|c@{}}
   \begin{matrix}
     I_{i-1} 
   \end{matrix} 
      & 0 & 0 \\
      \hline
    0 &\hspace{0.2cm} \begin{matrix}
   		N_i
   		\end{matrix}  & 0  \\
\hline
0 & 0 & I_{n-i-1}
\end{array} \right) \hspace*{0.2cm} \text{for} \hspace*{0.2cm} 1\leq i\leq n-1,$$ 
where $M_i, N_i \in \mathrm{GL}_k(\mathbb{C})$ with $k=m-n+2$ and $I_r$ is the $r\times r$ identity matrix. The representation $\rho$ is said to be homogeneous if all the matrices $M_i$ are equal and all the matrices $N_i$ are equal.
\end{definition}

The representations given in the following four definitions are famous examples of homogeneous local representations of the braid group $B_n$ to $\mathrm{GL}_m(\mathbb{C})$ in two cases: $m=n$ and $m=n+1$.
 
\begin{definition} \cite{W.B} \label{defBurau}
The Burau representation $\mu_B: B_n\longrightarrow \mathrm{GL}_n(\mathbb{C})$ is the representation defined by
$$\sigma_i\longmapsto \left( \begin{array}{c|@{}c|c@{}}
   \begin{matrix}
     I_{i-1} 
   \end{matrix} 
      & 0 & 0 \\
      \hline
    0 &\hspace{0.2cm} \begin{matrix}
   	1-t & t\\
   	1 & 0\\
\end{matrix}  & 0  \\
\hline
0 & 0 & I_{n-i-1}
\end{array} \right) \hspace*{0.2cm} \text{for} \hspace*{0.2cm} 1\leq i\leq n-1.$$ 
\end{definition}

\begin{definition} \cite{wada} \label{defwada}
Let $k$ be a non-zero integer. The Wada representation of type $1$, $\mu_w^{(k)}: B_n\longmapsto \mathrm{GL}_n(\mathbb{C})$, is the representation given by
$$\sigma_i\longrightarrow \left( \begin{array}{c|@{}c|c@{}}
   \begin{matrix}
     I_{i-1} 
   \end{matrix} 
      & 0 & 0 \\
      \hline
    0 &\hspace{0.2cm} \begin{matrix}
   	1-t^k & t^k\\
   	1 & 0\\
\end{matrix}  & 0  \\
\hline
0 & 0 & I_{n-i-1}
\end{array} \right) \hspace*{0.2cm} \text{for} \hspace*{0.2cm} 1\leq i\leq n-1.$$ 
\end{definition}

\begin{definition}\cite{D.T}
The standard representation $\mu_S: B_n \longrightarrow \mathrm{GL}_n(\mathbb{C})$ is the representation given by
$$\sigma_i \longmapsto \left( \begin{array}{c|@{}c|c@{}}
   \begin{matrix}
     I_{i-1} 
   \end{matrix} 
      & 0 & 0 \\
      \hline
    0 &\hspace{0.2cm} \begin{matrix}
   	0 & t\\
   	1 & 0\\
\end{matrix}  & 0  \\
\hline
0 & 0 & I_{n-i-1}
\end{array} \right) \hspace*{0.2cm} \text{for} \hspace*{0.2cm} 1\leq i \leq n-1.$$ 
\end{definition}

\begin{definition} \cite{V.B2016} \label{defF}
The $F$-representation $\mu_F: B_n \longrightarrow \mathrm{GL}_{n+1}(\mathbb{C})$ is the representation defined by
$$\sigma_i\longmapsto \left( \begin{array}{c|@{}c|c@{}}
   \begin{matrix}
     I_{i-1} 
   \end{matrix} 
      & 0 & 0 \\
      \hline
    0 &\hspace{0.2cm} \begin{matrix}
   		1 & 1 & 0 \\
   		0 &  -t & 0 \\   		
   		0 &  t & 1 \\
   		\end{matrix}  & 0  \\
\hline
0 & 0 & I_{n-i-1}
\end{array} \right) \hspace*{0.2cm} \text{for} \hspace*{0.2cm} 1\leq i\leq n-1.$$ 
\end{definition}

\vspace*{0.1cm}

For more information on the irreducibility of these four representations of $B_n$, see \cite{E.F}, \cite{I.S}, \cite{A1}, and \cite{M.N2024} respectively.

\vspace*{0.1cm}

We now propose another type of representation of $SB_n$ known as the $\Phi$-type extension representation. We start by presenting a proposition made by Bardakov, Chbili, and Kozlovskaya in \cite{13}.

\vspace*{0.1cm}

\begin{proposition} \cite{13}
Let $\mu: B_n \longrightarrow \mathrm{GL}_n(\mathbb{C})$ be a representation of $B_n$ and let $a,b,c \in \mathbb{C}$ with $abc\neq 0$. Then, the map $\Phi_{a,b,c}: SB_n\longrightarrow \mathrm{GL}_n(\mathbb{C})$ which acts on the generators of $SB_n$ by the rules
\begin{align*}
&\Phi_{a,b,c}(\sigma_i^{\pm 1})=\mu(\sigma_i^{\pm 1}) \hspace{0.2cm} \text{and} \hspace{0.2cm} \Phi_{a,b,c}(\tau_i)=a\mu(\sigma_i)+b\mu(\sigma_i^{-1})+cI_n,\hspace{0.2cm} i=1,2,\ldots,n-1,
\end{align*}
defines a representation of $SB_n$.
\end{proposition}

\vspace*{0.1cm}

Now, we introduce the following proposition regarding the relation between the irreducibility of a representation of $B_n$ and its $\Phi$-type extensions to $SB_n$. First, we need the following lemma

\begin{lemma} \label{lemmaaaa}
For $n\in \mathbb{N}^*$, let $\rho: G \longrightarrow \mathrm{GL}_n(V)$ be a representation, where $G$ is a group and $V$ is a vector space, and let $H$ be a subgroup of $G$. If the restriction of $\rho$ to $H$ is irreducible, then $\rho$ is irreducible. 
\end{lemma}

\begin{proof}
Suppose, to get a contradiction, that $\rho$ is reducible; then, there exists a non-trivial invariant subspace $S$ of $V$ such that $\rho(g)(s)\in S$ for every $s\in S$ and $g\in G$. This implies that $\rho(h)(s)\in S$ for every $s\in S$ and $h\in H\subset G$. Hence, $S$ is an invariant subspace of $V$ with respect to the restriction of $\rho$ to $H$, which is a contradiction.
\end{proof}

\begin{proposition}
Let $\mu: B_n \longrightarrow \mathrm{GL}_n(\mathbb{C})$ be a representation of $B_n$ and $\Phi_{a,b,c}$ be a $\Phi$-type extension representation of $\mu$ to $SB_n$, where $a,b,c \in \mathbb{C}$ with $abc\neq 0$. Then, $\mu$ is irreducible if and only if $\Phi_{a,b,c}$ is irreducible. 
\end{proposition}

\begin{proof}
The necessary condition follows directly. Indeed, if $\mu$ is irreducible, then by Lemma \ref{lemmaaaa} we obtain that the representation $\Phi_{a,b,c}$ is also irreducible.

For the sufficient condition, suppose that $\mu$ is reducible. Then there exists a non-trivial proper subspace $U \subset \mathbb{C}^n$ that is invariant under $\mu$. This means that for every $u \in U$ and for every generator $\sigma_i$ with $1 \leq i \leq n-1$, we have $\mu(\sigma_i)(u) \in U.$ Since $U$ is a linear subspace and $\mu(\sigma_i)$ is invertible, it follows that $\mu(\sigma_i^{-1})(u) \in U \text{ for all } u \in U \text{ and } 1 \leq i \leq n-1.$ Moreover, the identity map $I_n$ clearly preserves $U$, so for all $u \in U$ we also have $I_n(u) \in U$. As the representation $\Phi_{a,b,c}$ is defined in terms of $\mu(\sigma_i)$, $\mu(\sigma_i^{-1})$, and $I_n$, each of which leaves $U$ invariant, we conclude that $\Phi_{a,b,c}(\tau_i)(u) \in U \text{ for all } u \in U \text{ and } 1 \le i \le n-1.$ Therefore, $U$ is an invariant subspace for $\Phi_{a,b,c}$, and hence $\Phi_{a,b,c}$ is reducible.
\end{proof}

\vspace*{0.15cm}

\section{On the Irreduciblity of Complex Representations of $SB_2$ of Degree $2$} 

In this section, we aim to classify the forms of all irreducible representations $\rho: SB_2 \longrightarrow \mathrm{GL}_2(\mathbb{C})$. First, in the following theorem, we classify all representations $\rho: SB_2 \longrightarrow \mathrm{GL}_2(\mathbb{C})$.

\begin{theorem} \label{ThSB2}
Let $\rho: SB_2 \longrightarrow \mathrm{GL}_2(\mathbb{C})$ be a representation. Then, $\rho$ is equivalent to one of the following four representations.
\begin{itemize}
\item[(1)] $\rho_1: SB_2 \longrightarrow \mathrm{GL}_2(\mathbb{C})$ such that
$$\rho_1(\sigma_1)=\left( \begin{matrix}
w & x \\
y & z
\end{matrix}\right) \text{ and }\ 
\rho_1(\tau_1)=\left( \begin{matrix}
a & b \\
\frac{by}{x} & \frac{ax-bw+bz}{x}
\end{matrix} \right),$$
where $a,b,w,x,y,z \in \mathbb{C}, a^2x-abw+abz\neq b^2y, wz\neq xy, x\neq 0.$
\vspace*{0.1cm}
\item[(2)] $\rho_2: SB_2 \longrightarrow \mathrm{GL}_2(\mathbb{C})$ such that
$$\rho_2(\sigma_1)=\left( \begin{matrix}
w & 0 \\
y & z
\end{matrix}\right) \text{ and }\ 
\rho_2(\tau_1)=\left( \begin{matrix}
a & 0 \\
c & \frac{ay-cw+cz}{y}
\end{matrix} \right),$$
where $a,c,w,y,z \in \mathbb{C},a\neq 0, w\neq 0, y\neq 0, z\neq 0, ay-cw+cz\neq 0.$
\vspace*{0.1cm}
\item[(3)] $\rho_3: SB_2 \longrightarrow \mathrm{GL}_2(\mathbb{C})$ such that
$$\rho_3(\sigma_1)=\left( \begin{matrix}
w & 0 \\
0 & w
\end{matrix}\right) \text{ and }\ 
\rho_3(\tau_1)=\left( \begin{matrix}
a & b \\
c & d
\end{matrix} \right),$$
where $a,b,c,d,w \in \mathbb{C}, ad\neq bc, w\neq 0$.
\vspace*{0.1cm}
\item[(4)] $\rho_4: SB_2 \longrightarrow \mathrm{GL}_2(\mathbb{C})$ such that
$$\rho_4(\sigma_1)=\left( \begin{matrix}
w & 0 \\
0 & z
\end{matrix}\right) \text{ and }\ 
\rho_4(\tau_1)=\left( \begin{matrix}
a & 0 \\
0 & d
\end{matrix} \right),$$
where $a,d,w,z \in \mathbb{C}^*.$
\end{itemize}
\end{theorem}

\begin{proof}
Set $\rho(\sigma_1)=\left( \begin{matrix}
w & x \\
y & z
\end{matrix}\right) \text{ and } \rho(\tau_1)=\left( \begin{matrix}
a & b \\
c & d
\end{matrix} \right),$ where $a,b,c,d,w,x,y,z \in \mathbb{C},$ with $ad\neq bc$ and $wz\neq xy.$ Recall that $SB_2$ is generated by the two generators $\sigma_1$ and $\tau_1$ with the only relation: $\sigma_1\tau_1=\tau_1\sigma_1$. So, we get that $\rho(\sigma_1)\rho(\tau_1)=\rho(\tau_1)\rho(\sigma_1)$. Applying this equation, we obtain the following system of three equations and eight unknowns.
\begin{equation} \label{eq1}
-cx+by=0,
\end{equation}
\begin{equation}\label{eq2}
-bw+ax-dx+bz=0,
\end{equation}
\begin{equation}\label{eq3}
cw-ay+dy-cz=0.
\end{equation}
We solve this system of equations by considering the following cases.
\begin{itemize}
\item[(a)] If $x=0$, then, by Equation (\ref{eq1}), we get $b=0$ or $y=0$.
\begin{itemize}
\item[•]If $b=0$ and $y\neq 0$, then, using Equation (\ref{eq3}), we can see that $\rho$ is equivalent to $\rho_2$.
\item[•]If $b\neq 0$ and $y=0$, then, using Equation (\ref{eq2}), we can see that $\rho$ is equivalent to $\rho_3$.
\item[•]If $b=y=0$, then, using Equation (\ref{eq3}), we can see that $\rho$ is equivalent to a special case of $\rho_3$ if $c\neq 0$ and $\rho$ is equivalent to $\rho_4$ if $c=0$.
\end{itemize}
\item[(b)] If $x\neq 0$, then, using the three equations above, we can see that $\rho$ is equivalent to $\rho_1$.
\end{itemize}
\end{proof}

Now, we study the irreducibility of the representations $\rho_i, 1\leq i \leq 4$, determined in Theorem \ref{ThSB2}. First, we need the following lemma.

\begin{lemma} \label{Lemma1}
Let $\rho: SB_2 \longrightarrow \mathrm{GL}_2(\mathbb{C})$ be a representation. Then, $\rho$ is reducible if and only if $\rho(\sigma_1)$ and $\rho(\tau_1)$ have a common eigenvector.
\end{lemma}

\begin{proof}
Trivial.
\end{proof}

\begin{theorem} \label{ThSB22}
The representations $\rho_i: SB_2 \longrightarrow \mathrm{GL}_2(\mathbb{C})$, $1\leq i \leq 4$, determined in Theorem \ref{ThSB2} are reducible.
\end{theorem}

\begin{proof}
We consider each case separately.
\begin{itemize}
\item[(1)] In the case of $\rho_1$, we have the following two subcases.\vspace*{0.1cm}
\begin{itemize}
\item[•] If $y=0$, then, the column vector $(1,0)^T$ is a common eigenvector of $\rho_1(\sigma_1)$ and $\rho_1(\tau_1)$. Thus, by Lemma \ref{Lemma1}, $\rho_1$ is reducible.\vspace*{0.1cm}
\item[•] If $y\neq 0$, then, the column vector $\left(\frac{w-z-\sqrt{w^2+4xy-2wz+z^2}}{2y},1\right)^T$ is a common eigenvector of $\rho_1(\sigma_1)$ and $\rho_1(\tau_1)$. Thus, by Lemma \ref{Lemma1}, $\rho_1$ is reducible.\vspace*{0.1cm}
\end{itemize}
\item[(2)] In the case of $\rho_2$, we have the column vector $(0,1)^T$ is a common eigenvector of $\rho_2(\sigma_1)$ and $\rho_2(\tau_1)$. Thus, by Lemma \ref{Lemma1}, $\rho_2$ is reducible.\vspace*{0.1cm}
\item[(3)] In the case of $\rho_3$, every nonzero eigenvector of $\rho_3(\tau_1)$ is also an eigenvector of $\rho_3(\sigma_1)$, since $\rho_3(\sigma_1)$ is a scalar matrix. Therefore, by Lemma \ref{Lemma1}, $\rho_3$ is reducible.\vspace*{0.1cm}
\item[(4)] Finally, in the case of $\rho_4$, we have the column vector $(1,0)^T$ is a common eigenvector of $\rho_4(\sigma_1)$ and $\rho_4(\tau_1)$. Thus, by Lemma \ref{Lemma1}, $\rho_4$ is reducible.
\end{itemize}
\end{proof}

As a conclusion of this section, we show that all representations of the form $\rho: SB_2 \longrightarrow \mathrm{GL}_2(\mathbb{C})$ are reducible.

\section{Irreducible Complex Representations of $SB_3$ of Degrees $2$ and $3$} 

In this section, we aim to classify the forms of the irreducible representations $\rho: SB_3 \longrightarrow \mathrm{GL}_2(\mathbb{C})$ and $\rho: SB_3 \longrightarrow \mathrm{GL}_3(\mathbb{C})$.

\subsection{Irreducible representations of $SB_3$ of degree $2$}
We classify the forms of all irreducible representations $\rho: SB_3 \longrightarrow \mathrm{GL}_2(\mathbb{C})$. First, we introduce a theorem which was proved by I. Tuba and H. Wenzl in \cite{I.Tuba}. They classified all irreducible representations $\mu: B_3 \longrightarrow \mathrm{GL}_2(\mathbb{C})$.
 
\begin{theorem}\cite{I.Tuba} \label{irred2}
Let $\mu: B_3 \longrightarrow \mathrm{GL}_2(\mathbb{C})$ be an irreducible representation of $B_3$, then, $\mu$ acts on the generators of $B_3$, $\sigma_1$ and $\sigma_2$, as follows.
$$\mu(\sigma_1) =\left( \begin{array}{@{}c@{}}
  \begin{matrix}
   		\lambda_1 & \lambda_1 \\
   		0 & \lambda_2 \\
   		\end{matrix}
\end{array} \right) \text{ and \ } 
\mu(\sigma_2) =\left( \begin{array}{@{}c@{}}
  \begin{matrix}
   		\lambda_2 & 0 \\
   		-\lambda_2 & \lambda_1 \\
   		\end{matrix}
\end{array} \right),$$
where $\lambda_1,\lambda_2 \in \mathbb{C}^*$ with $\lambda_1^2+\lambda_2^2-\lambda_1\lambda_2\neq 0$.
\end{theorem}

Now, we extend all irreducible representations $\mu: B_3 \longrightarrow \mathrm{GL}_2(\mathbb{C})$ to $SB_3$.

\begin{theorem} \label{Theooo}
Let $\rho: SB_3 \longrightarrow \mathrm{GL}_2(\mathbb{C})$ be a representation of $SB_3$ that extends the representation $\mu$ given in Theorem \ref{irred2}, then, $\rho$ acts on the generators of $SB_3$, $\sigma_1, \sigma_2, \tau_1$ and $\tau_2$, as follows.
$$\rho(\sigma_1)=\mu(\sigma_1), \rho(\sigma_2)=\mu(\sigma_2), \rho(\tau_1) =\left( \begin{array}{@{}c@{}}
  \begin{matrix}
   		a_1 & b_1 \\
   		0 & d_1 \\
   		\end{matrix}
\end{array} \right) \text{ and } \rho(\tau_2) =\left( \begin{array}{@{}c@{}}
  \begin{matrix}
   		a_2 & 0 \\
   		c_2 & d_2 \\
   		\end{matrix}
\end{array} \right),$$
where $a_1, b_1, d_1, a_2, c_2, d_2, \lambda_1, \lambda_2 \in \mathbb{C}, a_1\neq 0, d_1 \neq 0, a_2\neq 0, d_2\neq 0, \lambda_1\neq 0, \lambda_2\neq 0, \lambda_1^2+\lambda_2^2-\lambda_1\lambda_2\neq 0, d_1=a_2=a_1-\frac{b_1(\lambda_1-\lambda_2)}{\lambda_1}, c_2=-\frac{b_1\lambda_2}{\lambda_1}, d_2=a_1$.
\end{theorem}

\begin{proof}
As $\rho$ is an extension of $\mu$, it follows that $\rho(\sigma_1)=\mu(\sigma_1)$ and $\rho(\sigma_2)=\mu(\sigma_2)$. Now, set $\rho(\tau_1)=\left( \begin{matrix}
a_1 & b_1 \\
c_1 & d_1
\end{matrix}\right) \text{ and } \rho(\tau_2)=\left( \begin{matrix}
a_2 & b_2 \\
c_2 & d_2
\end{matrix} \right),$ where $a_1,b_1,c_1,d_1,a_2,b_2,c_2,d_2 \in \mathbb{C},$ with $a_1d_1\neq b_1c_1$ and $a_2d_2\neq b_2c_2.$ The relations that we have between the generators of $SB_3$ are: $\tau_1\sigma_1=\sigma_1\tau_1, \tau_2\sigma_2=\sigma_2\tau_2, \sigma_1\sigma_2\tau_1=\tau_2\sigma_1\sigma_2$ and $\sigma_2\sigma_1\tau_2=\tau_1\sigma_2\sigma_1$. Applying these relations to the images of $\sigma_1,\sigma_2,\tau_1$ and $\tau_2$ under $\rho$, and using a similar strategy as in the proof of Theorem \ref{ThSB2}, we get our required result.
\end{proof}

\begin{corollary}
The representations $\rho: SB_3 \longrightarrow \mathrm{GL}_2(\mathbb{C})$ determined in Theorem \ref{Theooo} are irreducible.
\end{corollary}

\begin{proof}
The proof is a consequence of Lemma \ref{lemmaaaa} and Theorem \ref{irred2}.
\end{proof}

Now, we show that the only irreducible representations $\rho: SB_3 \longrightarrow \mathrm{GL}_2(\mathbb{C})$ are those that are equivalent to representations determined in Theorem \ref{Theooo}.

\begin{theorem}
Consider a representation $\rho: SB_3 \longrightarrow \mathrm{GL}_2(\mathbb{C})$. If the restriction of $\rho$ to $B_3$ is reducible, then $\rho$ itself is reducible.
\end{theorem}

\begin{proof}
Suppose that the restriction of $\rho$ to $B_3$ is reducible, and suppose, to get a contradiction, that $\rho$ itself is irreducible. As the restriction of $\rho$ to $B_3$ is reducible, there exists a non-trivial subspace $V$ of $\mathbb{C}^2$ that is invariant under $\rho(\sigma_1)$ and $\rho(\sigma_2)$. We have $\dim(V)=1$ since $V$ is non-trivial, and so $V=<v>$ for a nonzero vector $v \in \mathbb{C}^2$. We have
$$\rho(\sigma_1)(v)=\lambda_1v$$
and
$$\rho(\sigma_2)(v)=\lambda_2v.$$
But we know that
$$\sigma_1\sigma_2\sigma_1=\sigma_2\sigma_1\sigma_2,$$
which implies that
$$\rho(\sigma_1)\rho(\sigma_2)\rho(\sigma_1)=\rho(\sigma_2)\rho(\sigma_1)\rho(\sigma_2),$$
and so 
$$\rho(\sigma_1)\rho(\sigma_2)\rho(\sigma_1)(v)=\rho(\sigma_2)\rho(\sigma_1)\rho(\sigma_2)(v),$$
which gives that
$\lambda_1^2\lambda_2=\lambda_1\lambda_2^2$, and so $\lambda_1=\lambda_2$ as both matrices are invertible. We set in the following $\lambda= \lambda_1=\lambda_2$. On the other hand, we have 
$$\sigma_1\tau_1=\tau_1\sigma_1,$$
and then
$$\rho(\sigma_1)\rho(\tau_1)=\rho(\tau_1)\rho(\sigma_1),$$
which implies that
$$\rho(\sigma_1)\rho(\tau_1)(v)=\rho(\tau_1)\rho(\sigma_1)(v),$$
and so
$$\rho(\sigma_1)\rho(\tau_1)(v)=\rho(\tau_1)\lambda_1(v)=\lambda_1\rho(\tau_1)(v).$$
This gives that $\rho(\tau_1)(v)$ is an eigenvector of $\rho(\sigma_1)$ with eigenvalue $\lambda$. Similarly, $\rho(\tau_2)(v)$ is an eigenvector of $\rho(\sigma_2)$ with eigenvalue $\lambda$. As $\rho$ is irreducible and the vector $v$ is invariant under $\rho(\sigma_1)$ and $\rho(\sigma_2)$, we have that $\rho(\tau_1)(v)$ and $v$ are linearly independent or $\rho(\tau_2)(v)$ and $v$ are linearly independent. Without loss of generality, suppose that $\rho(\tau_1)(v)$ and $v$ are linearly independent and let $P$ be the $2\times 2$ invertible matrix whose columns are $\rho(\tau_1)(v)$ and $v$. We now consider an equivalent representation $\rho'$ to $\rho$ defined as 
$$\rho'(g)=P^{-1}\rho(g)P$$ 
for all $g \in SB_3$. Hence, we can see that $\rho'(\sigma_1)=\lambda I_2$. In addition, we can see that $\rho'(\sigma_2)=\lambda I_2$ using the relation $\sigma_1\sigma_2\sigma_1=\sigma_2\sigma_1\sigma_2$. But we have the relation
$$\sigma_2\sigma_1\tau_2=\tau_1\sigma_2\sigma_1,$$
which gives that
$$\rho'(\sigma_2)\rho'(\sigma_1)\rho'(\tau_2)=\rho'(\tau_1)\rho'(\sigma_2)\rho'(\sigma_1),$$
and so
$$\lambda I_2 \lambda I_2 \rho'(\tau_2)=\rho'(\tau_1)\lambda I_2 \lambda I_2,$$
and then 
$$\rho'(\tau_2)=\rho'(\tau_1).$$
So, we have that $\rho'(\tau_2)$ and $\rho'(\tau_1)$ have a common eigenvector, since they are equal, and $\rho'(\sigma_1)$ and $\rho'(\sigma_2)$ are both constant matrices, which gives that $\rho$ is reducible, a contradiction. Therefore, $\rho$ is reducible, as required.
\end{proof}

\begin{corollary} \label{cor}
The only irreducible representations $\rho: SB_3 \longrightarrow \mathrm{GL}_2(\mathbb{C})$ are the equivalent representations to $\rho: SB_3 \longrightarrow \mathrm{GL}_2(\mathbb{C})$ such that
$$\rho(\sigma_1)=\mu(\sigma_1), \rho(\sigma_2)=\mu(\sigma_2), \rho(\tau_1) =\left( \begin{array}{@{}c@{}}
  \begin{matrix}
   		a_1 & b_1 \\
   		0 & d_1 \\
   		\end{matrix}
\end{array} \right) \text{ and } \rho(\tau_2) =\left( \begin{array}{@{}c@{}}
  \begin{matrix}
   		a_2 & 0 \\
   		c_2 & d_2 \\
   		\end{matrix}
\end{array} \right),$$
where $a_1, b_1, d_1, a_2, c_2, d_2, \lambda_1, \lambda_2 \in \mathbb{C}, a_1\neq 0, d_1 \neq 0, a_2\neq 0, d_2\neq 0, \lambda_1\neq 0, \lambda_2\neq 0, \lambda_1^2+\lambda_2^2-\lambda_1\lambda_2\neq 0, d_1=a_2=a_1-\frac{b_1(\lambda_1-\lambda_2)}{\lambda_1}, c_2=-\frac{b_1\lambda_2}{\lambda_1}, d_2=a_1$.
\end{corollary}

\subsection{Irreducible representations of $SB_3$ of degree $3$}
Now, we classify the forms of some irreducible representations $\rho: SB_3 \longrightarrow \mathrm{GL}_3(\mathbb{C})$. Again, we introduce a theorem which was proved by I. Tuba and H. Wenzl in \cite{I.Tuba}. They also classified all irreducible representations $\mu: B_3 \longrightarrow \mathrm{GL}_3(\mathbb{C})$.

\begin{theorem}\cite{I.Tuba} \label{irred3}
Let $\mu: B_3 \longrightarrow \mathrm{GL}_3(\mathbb{C})$ be an irreducible representation of $B_3$, then, $\mu$ acts on the generators of $B_3$, $\sigma_1$ and $\sigma_2$, as follows.
$$\mu(\sigma_1) =\left( \begin{array}{@{}c@{}}
  \begin{matrix}
   		\lambda_1 & \lambda_1\lambda_2^{-1}\lambda_3+\lambda_2 & \lambda_2 \\
   		0 & \lambda_2 & \lambda_2 \\
   		0 & 0 & \lambda_3
   		\end{matrix}
\end{array} \right) \text{ and \ } 
\mu(\sigma_2) =\left( \begin{array}{@{}c@{}}
  \begin{matrix}
		\lambda_3 & 0 & 0 \\
   		-\lambda_2 & \lambda_2 & 0 \\
   		\lambda_2 & -\lambda_1\lambda_2^{-1}\lambda_3-\lambda_2 & \lambda_1
   		\end{matrix}
\end{array} \right),$$
where $\lambda_1,\lambda_2,\lambda_3 \in \mathbb{C}^*$ with $(\lambda_1^2+\lambda_2\lambda_3)(\lambda_2^2+\lambda_1\lambda_3)(\lambda_3^2+\lambda_1\lambda_2)\neq 0$.
\end{theorem}

Now, we extend all irreducible representations $\mu: B_3 \longrightarrow \mathrm{GL}_3(\mathbb{C})$ to $SB_3$.

\begin{theorem} \label{Theooo1}
Let $\rho: SB_3 \longrightarrow \mathrm{GL}_3(\mathbb{C})$ be a representation of $SB_3$ that extends the representation $\mu$ given in Theorem \ref{irred3}, then, $\rho$ acts on the generators of $SB_3$, $\sigma_1, \sigma_2, \tau_1$ and $\tau_2$, as follows.
$$\rho(\sigma_1)=\mu(\sigma_1), \rho(\sigma_2)=\mu(\sigma_2), \rho(\tau_1) =\left( \begin{array}{@{}c@{}}
  \begin{matrix}
   		a_1 & b_1 & c_1 \\
   		0 & e_1 & f_1\\
   		0 & 0 & i_1
   		\end{matrix}
\end{array} \right) \text{ and } \rho(\tau_2) =\left( \begin{array}{@{}c@{}}
  \begin{matrix}
   		a_2 & 0 & 0 \\
   		d_2 & e_2 & 0\\
   		g_2 & h_2 & i_2
   		\end{matrix}
\end{array} \right),$$
where $a_1,b_1,c_1,e_1,f_1, i_1, a_2, d_2, e_2, g_2, h_2, i_2, \lambda_1, \lambda_2, \lambda_3 \in \mathbb{C}, a_1\neq 0,e_1\neq 0, i_1\neq 0, a_2\neq 0, e_2\neq 0, i_2\neq 0, \lambda_1 \neq 0, \lambda_2 \neq 0, \lambda_3\neq 0$ with $(\lambda_1^2+\lambda_2\lambda_3)(\lambda_2^2+\lambda_1\lambda_3)(\lambda_3^2+\lambda_1\lambda_2)\neq 0, \lambda_2\neq -\lambda_3, a_1=i_2=\frac{c_1\lambda_2(\lambda_1-\lambda_2)(\lambda_1-\lambda_3)+e_1\lambda_1\lambda_2(\lambda_2+\lambda_3)+f_1\lambda_3(\lambda_1-\lambda_2)(\lambda_1+\lambda_2)}{\lambda_1\lambda_2(\lambda_2+\lambda_3)},\\ b_1=-h_2=\frac{(\lambda_1\lambda_3+\lambda_2^2)(c_1\lambda_2(\lambda_1-\lambda_3)+f_1\lambda_3(\lambda_1+\lambda_2))}{\lambda_1\lambda_2^2(\lambda_2+\lambda_3)},$ $ i_1=a_2=e_1+f_1(\frac{\lambda_3}{\lambda_2}-1), d_2=-f_1, e_2=e_1, g_2=c_1.$
\end{theorem}

\begin{proof}
The proof is similar to the proof of Theorem \ref{Theooo}.
\end{proof}

\begin{corollary}
The representations $\rho: SB_3 \longrightarrow \mathrm{GL}_3(\mathbb{C})$ determined in Theorem \ref{Theooo1} are irreducible.
\end{corollary}

\begin{proof}
The proof is a consequence of Lemma \ref{lemmaaaa} and Theorem \ref{irred3}.
\end{proof}

We end this section with the following question.

\begin{question} \label{quee}
Let $\rho: SB_3 \longrightarrow \mathrm{GL}_3(\mathbb{C})$ be an irreducible representation. What are the possible forms of $\rho$?
\end{question}

\vspace*{0.15cm}

\section{On the Irreducibility of Complex Homogeneous Local Representations of $SB_n$} 

In \cite{Mikha}, Y. Mikhalchishina classified all homogeneous local representations of $B_n$ for all $n\geq 3$. Moreover, in \cite{M.Ch}, M. Chreif and M. Dally studied the irreducibility of these representations in some cases. In \cite{Ma}, T. Mayassi and M. Nasser classified all homogeneous local representations of $SB_n$ for all $n\geq 2$. In this section, we study the irreducibility of these homogeneous local representations of $SB_n$ in most cases.

\vspace*{0.1cm}

First, we introduce, in the following four theorems, previous results done in \cite{Mikha}, \cite{M.Ch}, and \cite{Ma}.

\begin{theorem} \cite{Mikha} \label{ThMi}
Consider $n\geq 3$ and let $\mu: B_n \longrightarrow \mathrm{GL}_n(\mathbb{C})$ be a nontrivial homogeneous local representation of $B_n$. Then, $\mu$ is equivalent to one of the following three representations.
\begin{itemize}
\item[(1)] $\mu_1: B_n \longrightarrow \mathrm{GL}_n(\mathbb{C}) \hspace*{0.15cm} \text{such that } 
\mu_1(\sigma_i) =\left( \begin{array}{c|@{}c|c@{}}
   \begin{matrix}
     I_{i-1} 
   \end{matrix} 
      & 0 & 0 \\
      \hline
    0 &\hspace{0.2cm} \begin{matrix}
   		a & \frac{1-a}{c}\\
   		c & 0
   		\end{matrix}  & 0  \\
\hline
0 & 0 & I_{n-i-1}
\end{array} \right), \\ \text{ where } c \neq 0, a\neq 1, \hspace*{0.15cm} \text{for} \hspace*{0.2cm} 1\leq i\leq n-1.$
\item[(2)] $\mu_2: B_n \longrightarrow \mathrm{GL}_n(\mathbb{C}) \hspace*{0.15cm} \text{such that }
\mu_2(\sigma_i) =\left( \begin{array}{c|@{}c|c@{}}
   \begin{matrix}
     I_{i-1} 
   \end{matrix} 
      & 0 & 0 \\
      \hline
    0 &\hspace{0.2cm} \begin{matrix}
   		0 & \frac{1-d}{c}\\
   		c & d
   		\end{matrix}  & 0  \\
\hline
0 & 0 & I_{n-i-1}
\end{array} \right),\\ \text{ where } c\neq 0, d\neq 1 \hspace*{0.2cm} \text{for} \hspace*{0.2cm} 1\leq i\leq n-1.$
\item[(3)] $\mu_3: B_n \longrightarrow \mathrm{GL}_n(\mathbb{C}) \hspace*{0.15cm} \text{such that } 
\mu_3(\sigma_i) =\left( \begin{array}{c|@{}c|c@{}}
   \begin{matrix}
     I_{i-1} 
   \end{matrix} 
      & 0 & 0 \\
      \hline
    0 &\hspace{0.2cm} \begin{matrix}
   		0 & b\\
   		c & 0
   		\end{matrix}  & 0  \\
\hline
0 & 0 & I_{n-i-1}
\end{array} \right),\\ \text{ where } bc\neq 0 \hspace*{0.2cm} \text{for} \hspace*{0.2cm} 1\leq i\leq n-1.$
\end{itemize}
\end{theorem}

\begin{theorem} \cite{M.Ch} \label{ThCh1}
The homogeneous local representations of the form $\mu_1$ and $\mu_2$ defined in Theorem \ref{ThMi} are reducible.
\end{theorem}

\begin{theorem} \cite{M.Ch} \label{ThCh2}
The homogeneous local representations of the form $\mu_3$ defined in Theorem \ref{ThMi} are irreducible if and only if $bc\neq 1$, for $n\geq 3$.
\end{theorem}

\begin{theorem} \cite{Ma} \label{TaMa}
Consider $n\geq 3$ and let $\rho: SB_n \longrightarrow \mathrm{GL}_n(\mathbb{C})$ be a nontrivial homogeneous local representation of $SB_n$. Then, $\rho$ is equivalent to one of the following three representations. 
\begin{itemize}
\item[(1)] $\rho_1: SB_n \longrightarrow \mathrm{GL}_n(\mathbb{C}) \hspace*{0.15cm} \text{such that}\\ \hspace*{0.15cm} \rho_1(\sigma_i) =\mu_1(\sigma_i) \hspace*{0.15cm} \text{and} \hspace*{0.15cm} \rho_1(\tau_i) =\left( \begin{array}{c|@{}c|c@{}}
   \begin{matrix}
     I_{i-1} 
   \end{matrix} 
      & 0 & 0 \\
      \hline
    0 &\hspace{0.2cm} \begin{matrix}
   		 1-(1-a)(1-t) & \frac{1-a}{c}(1-t)\\
   		c(1-t) & t
   		\end{matrix}  & 0  \\
\hline
0 & 0 & I_{n-i-1}
\end{array} \right), \\
\text{where} \hspace*{0.15cm} a, c, t \in \mathbb{C},a\neq 0, c\neq 0, \hspace*{0.15cm} \text{for} \hspace*{0.15cm} 1\leq i\leq n-1.$
\item[(2)] $\rho_2:SB_n \longrightarrow \mathrm{GL}_n(\mathbb{C}) \hspace*{0.15cm} \text{such that}\\ \hspace*{0.15cm} \rho_2(\sigma_i)=\mu_2(\sigma_i) \hspace*{0.15cm} \text{and} \hspace*{0.15cm} \rho_2(\tau_i) =\left( \begin{array}{c|@{}c|c@{}}
   \begin{matrix}
     I_{i-1} 
   \end{matrix} 
      & 0 & 0 \\
      \hline
    0 &\hspace{0.2cm} \begin{matrix}
   		x & \frac{1-d}{c}(1-x)\\
   		c(1-x) & 1-(1-d)(1-x)
   		\end{matrix}  & 0  \\   		
\hline
0 & 0 & I_{n-i-1}
\end{array} \right), \\
\text{where} \hspace*{0.15cm} c, d, x \in \mathbb{C}, c\neq 0, d\neq 0, \hspace*{0.15cm} \text{for} \hspace*{0.15cm} 1\leq i\leq n-1.$ 
\item[(3)] $\rho_3:SB_n \longrightarrow \mathrm{GL}_n(\mathbb{C}) \hspace*{0.15cm} \text{such that}\\ \hspace*{0.15cm} \rho_3(\sigma_i)=\mu_3(\sigma_i) \hspace*{0.15cm}  \text{and} \hspace*{0.15cm} \rho_3(\tau_i) =\left( \begin{array}{c|@{}c|c@{}}
   \begin{matrix}
     I_{i-1} 
   \end{matrix} 
      & 0 & 0 \\
      \hline
    0 &\hspace{0.2cm} \begin{matrix}
   		x & y\\
   		\frac{cy}{b} & x
   		\end{matrix}  & 0  \\
\hline
0 & 0 & I_{n-i-1}
\end{array} \right), \\
\text{where} \hspace*{0.15cm} b, c, x, y \in \mathbb{C}, b\neq 0, c\neq 0, x^2-\frac{cy^2}{b}\ne 0, \hspace*{0.15cm} \text{for} \hspace*{0.15cm} 1\leq i\leq n-1.$ 
\end{itemize}
Here, $\mu_1, \mu_2,$ and $\mu_3$ are the representations given in Theorem \ref{ThMi}.
\end{theorem}

Now, we prove similar results as in Theorems \ref{ThCh1} and \ref{ThCh2} for the representations $\rho_1, \rho_2,$ and $\rho_3$ given in Theorem \ref{TaMa}. We follow the same strategy used by Chreif and Dally in  \cite{M.Ch}.

\begin{theorem} \label{Thhhhh}
Set $n\geq 3$. The homogeneous local representations of $SB_n$ of the form $\rho_1$ and $\rho_2$ defined in Theorem \ref{TaMa} are reducible.
\end{theorem}

\begin{proof}
We prove it just for $\rho_1$, and the proof will be similar for $\rho_2$. Consider the representation $\rho_1$ defined in Theorem \ref{TaMa}. Consider the diagonal matrix $P=\text{diag}(c^{1-n},c^{2-n},\ldots, c,1)$, where $\text{diag}(x_1,x_2,\dots,x_n)$ is a diagonal $n\times n$ matrix with $x_{ii}=x_i$. Consider the equivalent representation $\hat{\rho}_1$ of $\rho_1$ given by: $\hat{\rho}_1(\sigma_i)=P^{-1}\rho_1(\sigma_i)P$ and $\hat{\rho}_1(\tau_i)=P^{-1}\rho_1(\tau_i)P$ for all $1\leq i \leq n-1$. Direct computations imply that 
$$\hat{\rho}_1(\sigma_i) =\left( \begin{array}{c|@{}c|c@{}}
   \begin{matrix}
     I_{i-1} 
   \end{matrix} 
      & 0 & 0 \\
      \hline
    0 &\hspace{0.2cm} \begin{matrix}
   		a & 1-a\\
   		1 & 0
   		\end{matrix}  & 0  \\
\hline
0 & 0 & I_{n-i-1}
\end{array} \right)$$
and
$$\hat{\rho}_1(\tau_i) =\left( \begin{array}{c|@{}c|c@{}}
   \begin{matrix}
     I_{i-1} 
   \end{matrix} 
      & 0 & 0 \\
      \hline
    0 &\hspace{0.2cm} \begin{matrix}
   		 1-(1-a)(1-t) & (1-a)(1-t)\\
   		(1-t) & t
   		\end{matrix}  & 0  \\
\hline
0 & 0 & I_{n-i-1}
\end{array} \right),$$
$\text{where} \hspace*{0.15cm} a, c, t \in \mathbb{C},a\neq 0, c\neq 0, \hspace*{0.15cm} \text{for} \hspace*{0.15cm} 1\leq i\leq n-1$. Now, we can easily see that the column vector $(1,1,\ldots,1)^T$ is invariant under $\hat{\rho}_1(\sigma_i)$ and $\hat{\rho}_1(\tau_i)$, for all $1\leq i \leq n-1$. Therefore, $\hat{\rho}_1$ is reducible, and so $\rho_1$ is reducible.
\end{proof}

\begin{theorem}\label{Thlast}
Set $n\geq 3$. The homogeneous local representations of the form $\rho_3$ defined in Theorem \ref{TaMa} are irreducible if $bc\neq 1$. Moreover, if $bc=1$ then $\rho_3$ is irreducible if and only if $x+\frac{y}{b}\neq 1$.
\end{theorem}

\begin{proof}
Following the same procedure as in the proof of Theorem \ref{Thhhhh}, we can see that $\rho_3$ is equivalent to the representation $\hat{\rho}_3$, which is defined as follows.
$$\hat{\rho}_3(\sigma_i) =\left( \begin{array}{c|@{}c|c@{}}
   \begin{matrix}
     I_{i-1} 
   \end{matrix} 
      & 0 & 0 \\
      \hline
    0 &\hspace{0.2cm} \begin{matrix}
   		0 & bc\\
   		1 & 0
   		\end{matrix}  & 0  \\
\hline
0 & 0 & I_{n-i-1}
\end{array} \right)$$
and
$$\hat{\rho}_3(\tau_i) =\left( \begin{array}{c|@{}c|c@{}}
   \begin{matrix}
     I_{i-1} 
   \end{matrix} 
      & 0 & 0 \\
      \hline
    0 &\hspace{0.2cm} \begin{matrix}
   		 x & cy\\
   		\frac{y}{b} & x
   		\end{matrix}  & 0  \\
\hline
0 & 0 & I_{n-i-1}
\end{array} \right),$$
$\text{where} \hspace*{0.15cm} b, c, x, y \in \mathbb{C}, b\neq 0, c\neq 0, \hspace*{0.15cm} \text{for} \hspace*{0.15cm} 1\leq i\leq n-1$.

\vspace*{0.1cm}

If $bc\neq 1$ then the restriction of $\rho_3$ to $B_n$ is irreducible by Theorem \ref{ThCh2}. Thus, by Lemma \ref{lemmaaaa}, we get that $\rho_3$ is irreducible. Now, suppose that $bc=1$. Then, the representation $\hat{\rho}_3$ will be as follows.
$$\hat{\rho}_3(\sigma_i) =\left( \begin{array}{c|@{}c|c@{}}
   \begin{matrix}
     I_{i-1} 
   \end{matrix} 
      & 0 & 0 \\
      \hline
    0 &\hspace{0.2cm} \begin{matrix}
   		0 & 1\\
   		1 & 0
   		\end{matrix}  & 0  \\
\hline
0 & 0 & I_{n-i-1}
\end{array} \right)$$
and
$$\hat{\rho}_3(\tau_i) =\left( \begin{array}{c|@{}c|c@{}}
   \begin{matrix}
     I_{i-1} 
   \end{matrix} 
      & 0 & 0 \\
      \hline
    0 &\hspace{0.2cm} \begin{matrix}
   		 x & \frac{y}{b}\\
   		\frac{y}{b} & x
   		\end{matrix}  & 0  \\
\hline
0 & 0 & I_{n-i-1}
\end{array} \right),$$
$\text{where} \hspace*{0.15cm} b, x, y \in \mathbb{C}, b\neq 0, \hspace*{0.15cm} \text{for} \hspace*{0.15cm} 1\leq i\leq n-1$. We consider two cases in the following.
\begin{itemize}
\item[(1)] Suppose first that $x+\frac{y}{b}=1$. In this case, the column vector $(1,1,\ldots,1)^T$ is invariant under $\hat{\rho}_3(\sigma_i)$ and $\hat{\rho}_3(\tau_i)$, for all $1\leq i \leq n-1$. Therefore, $\hat{\rho}_3$ is reducible, and so $\rho_3$ is reducible.
\item[(2)] Now, suppose that $x+\frac{y}{b}\neq 1$. Suppose to get a contradiction that $\rho_3$ is reducible, then $\hat{\rho}_3$ will also be reducible. Let $V$ be a nontrivial vector subspace of $\mathbb{C}^n$ that is invariant under $\hat{\rho}_3$. We consider two subcases in the following.
\begin{itemize}
\item[•] The first subcase is when $V=<v>$, where $v=(1,1,\ldots,1)^T$. As $V$ is invariant under $\hat{\rho}_3$, we get $\hat{\rho}_3(\tau_1)(v)-v=(x+\frac{y}{b}-1)(1,1,0,\ldots,0)^T \in V$ with $x+\frac{y}{b}-1\neq 0$. So, the vector $(1,1,0,\ldots,0)^T \in V$, a contradiction.
\item[•] The second subcase is when $V\neq <v>$, where $v=(1,1,\ldots,1)^T$. In this case we may pick $v=(v_1,v_2,\ldots,v_n)^T\in V$ such that $v_i \neq v_{i+1}$ for some $1\leq i \leq n-1$. Without loss of generality suppose that $i=1$; that is $v_1\neq v_2$. We have $\hat{\rho}_3(\sigma_1)(v)-v=(v_2-v_1)(1,-1,0,\ldots,0)^T \in V$, with $v_2-v_1\neq 0$, which gives that $u=(1,-1,0,\ldots,0)^T \in V$. Then, we can see that $w=\hat{\rho}_3(\sigma_2)(u)-u=(0,1,-1,0,\ldots,0)^T \in V$. Thus, we have that $\hat{\rho}_3(\tau_2)(u)-u+(x-1)w=(1-x-\frac{y}{b})(0,0,1,0,\ldots,0)^T\in V$ with $(1-x-\frac{y}{b})\neq 0$, and so the column vector $e_3 \in V$, where $\{e_1,e_2,\ldots, e_n\}$ is the standard basis of $\mathbb{C}^n$. Now, as $e_3 \in V$, we get that $\hat{\rho}_3(\sigma_2)(e_3)=e_2 \in V$ and so  $\hat{\rho}_3(\sigma_1)(e_2)=e_1 \in V$. Similarly, we have that $\hat{\rho}_3(\sigma_3)(e_3)=e_4 \in V$ and so $\hat{\rho}_3(\sigma_4)(e_4)=e_5 \in V$. If we continue till the end, we get that $\hat{\rho}_3(\sigma_{n-1})(e_{n-1})=e_n \in V$. Therefore, $e_i \in V$ for all $1\leq i \leq n$, and so $V=\mathbb{C}^n$, a contradiction.
\end{itemize}
\end{itemize}
\end{proof}
  








\end{document}